\documentclass[11pt]{amsart}
\usepackage{amssymb, amsmath, amscd, eucal, rotating}
\usepackage{times,epsf,epsfig,color}

\usepackage{tikz-cd}

\input xy
 \xyoption{all}

\numberwithin{equation}{section}

\theoremstyle{plain}

\newtheorem{proposition}{Proposition}[section]
\newtheorem{theorem}[proposition]{Theorem}		
\newtheorem*{theorem*}{Theorem}		
\newtheorem{corollary}[proposition]{Corollary}

\theoremstyle{definition}

\newtheorem{remark}[proposition]{Remark}

\newcommand{\R}{\mathbb R}

\DeclareMathOperator{\grad}{grad}

\DeclareMathOperator{\Crit}{Crit}

\begin{document}


\title[Local behaviour of an analytic flow near an unstable set]{Local behaviour of the gradient flow of an analytic function near the unstable set of a critical point}

\author{Graeme Wilkin}
\address{Department of Mathematics,
National University of Singapore, 
Singapore 119076}
\email{graeme@nus.edu.sg}

\thanks{This research was partially supported by grant number R-146-000-264-114 from the National University of Singapore. The author also acknowledges support from NSF grants DMS 1107452, 1107263, 1107367 ``RNMS GEometric structures And Representation varieties'' (the GEAR Network).}

\date{\today}

\begin{abstract}
This paper extends previous work \cite{wilkin-singular-morse}, which shows that the main theorem of Morse theory holds for a large class of functions on singular spaces, where the function and the underlying singular space are required to satisfy the five conditions explained in detail in the introduction to \cite{wilkin-singular-morse}. The fourth of these conditions requires that the gradient flow of the function is well-behaved near the critical points, which is a very natural condition, but difficult to explicitly check for examples without a detailed knowledge of the flow. In this paper we prove a general result showing that the first three conditions always imply the fourth when the underlying space is locally compact. Moreover, if the function is proper and analytic then the first four conditions are all satisfied.
\end{abstract}




\maketitle


\thispagestyle{empty}

\baselineskip=16pt



\section{Introduction}

Morse theory determines a relationship between the topology of a manifold and the critical points of certain smooth functions on that manifold. The key result that makes this work is the \emph{main theorem of Morse theory}, which describes the homotopy type of the manifold in terms of a sequence of cell attachments, each of which is determined by the local data of the function around a particular critical point.

Morse theory and its generalisations have been very successful in proving the existence of critical points for a wide variety of applications. The original example was Morse's work on the existence of geodesics of arbitrarily large length on a sphere with any metric \cite[Ch. IX]{Morse96} (see also \cite{Milnor63}, \cite{Palais63} and \cite{Smale64} for more modern treatments) and many more subsequent applications (see for example \cite{Bott82}, \cite{Bott88} and the references therein).  Conversely, when the critical points are known and the Morse function is perfect, then Morse theory can be used to compute topological invariants of a manifold; for example Bott's work on the topology of Lie groups \cite{Bott56} and the work of Atiyah \& Bott \cite{AtiyahBott83} and Kirwan \cite{Kirwan84} on the topology of symplectic quotients. 

There are also many interesting applications and potential applications of these ideas to singular spaces. The Stratified Morse Theory of Goresky and MacPherson \cite{GoreskyMacPhersonSMT} is a very general theory in this direction with applications to the topology of algebraic and analytic varieties \cite[Part II]{GoreskyMacPhersonSMT} and the topology of complements of affine subspaces of $\mathbb{R}^n$ \cite[Part III]{GoreskyMacPhersonSMT}. One of the motivating questions behind \cite{wilkin-singular-morse} is the question of extending the methods of Atiyah \& Bott and Kirwan to the norm-square of a moment map on a non-smooth affine variety, where Goresky and MacPherson's theory does not necessarily apply. Theorem 1.2 of \cite{wilkin-singular-morse} shows that the main theorem of Morse theory does indeed work in this setting. 

In fact this theorem applies in greater generality to functions on singular spaces satisfying the five conditions of \cite{wilkin-singular-morse}, which we now recall. 


Let $M$ be a real analytic Riemannian manifold with the metric topology, and $Z \subset M$ a closed subspace. In the sequel we will use $\| x_1 - x_2 \|$ to denote the distance between two points $x_1, x_2 \in M$ with respect to the metric on $M$. Given a smooth function $f : M \rightarrow \mathbb{R}$, the gradient of $f$ is a well-defined vector field $\grad f(x)$ on $M$. Suppose that for each $x \in Z$ there exists an open interval $(-\varepsilon, \varepsilon)$ such that the flow $\varphi(x,t)$ of $-\grad f$ with initial condition $x$ exists, depends continuously on the initial conditions and $\varphi(x,t) \in Z$ for all $t \in (-\varepsilon, \varepsilon)$. For example, these conditions are satisfied when the flow is generated by a group action which preserves the subspace $Z$. Moment map flows form an important class of such examples, but here we do not restrict to this setting. 

A \emph{critical point} of $f : Z \rightarrow \mathbb{R}$ is then a fixed point of this flow. Let $\Crit_Z(f) \subset Z$ denote the subset of all critical points in $Z$.  Given a critical value $c \in \R$ and the associated critical set $C = \Crit_Z(f) \cap f^{-1}(c)$, define $W_C^+$ and $W_C^-$ to be the stable and unstable sets of $C$ with respect to the flow
\begin{align*}
W_C^+ & := \left\{ x \in Z \mid \lim_{t \rightarrow \infty} \varphi(x, t) \in C \right\} \\
W_C^- & := \left\{ x \in Z \mid \lim_{t \rightarrow - \infty} \varphi(x, t) \in C \right\} .
\end{align*}
Let $W_x^+$ and $W_x^-$ denote the analogous stable/unstable sets with respect to a specific critical point $x \in C$. Recall the following conditions on the function $f : Z \rightarrow \mathbb{R}$ and its flow $\varphi$ from \cite{wilkin-singular-morse}.

\begin{enumerate}

\item The critical values of $f$ are isolated.

\item (Compactness of the flow) 

\noindent For any regular values $a < b$ and any $x \in f^{-1}((a,b))$ either there exists $T_+ > 0$ such that $f\left( \varphi(x, T_+) \right) = a$ or $\lim_{t \rightarrow \infty} \varphi(x, t)$ exists in $f^{-1}((a, b))$. Similarly, either there exists $T_- < 0$ such that $f \left( \varphi(x, T_-) \right) = b$ or $\lim_{t \rightarrow -\infty} \varphi(x, t)$ exists in $f^{-1}((a,b))$.

\item $f$ is analytic.

\item (Local behaviour of the flow near the critical points) 

\noindent For each non-minimal critical point $x$, let $C = \Crit_Z(f) \cap f^{-1}(f(x))$. For each $a < f(x)$ such that there are no critical values in $[a, f(x))$ and for each neighbourhood $U \subset f^{-1}(a)$ of $W_x^- \cap f^{-1}(a)$ there exists a neighborhood $V$ of $x$ such that for each $y \in V \setminus W_C^+$ there exists $T > 0$ such that $\varphi(y, T) \in U$.

\item (Neighbourhood deformation retract property of the level sets of the unstable set) 

\noindent For each critical value $c$ let $C = \Crit_Z(f) \cap f^{-1}(c)$. Then there exists $\varepsilon > 0$ such that $W_C^- \cap f^{-1}(c-\varepsilon)$ has an open neighborhood $E \subset f^{-1}(c-\varepsilon)$ and a strong deformation retract $r : E \times [0,1] \rightarrow E$ of $E$ onto $W_C^- \cap f^{-1}(c-\varepsilon)$ such that (using $E_s$ to denote the image $r(E, s)$ for each $s \in [0,1]$) we have
\begin{enumerate}

\item $E_s$ is open in $f^{-1}(c-\varepsilon)$ for all $s \in [0,1)$,


\item $E_s = \bigcup_{t > s} E_t$ and $\overline{E}_s = \bigcap_{t < s} E_t$ for all $s \in (0,1)$.

\end{enumerate}

\end{enumerate}

Theorem 1.1 of \cite{wilkin-singular-morse} shows that the main theorem of Morse theory holds if all of the above conditions are satisfied. Moreover, \cite[Thm. 1.2]{wilkin-singular-morse} shows that the theory is nonempty, in the sense that it is satisfied for a large class of interesting examples, namely when $f$ is the norm-square of a moment map on an affine variety.

The intuition behind Condition 4 is that if the flow begins close to a critical point $x$ then it should not wander far from the unstable set $W_x^-$. This is a very natural condition, however to check that it holds for a particular example requires some detailed knowledge of the flow; for example when $f : Z \rightarrow \mathbb{R}$ is the norm square of a moment map, then Kirwan \cite[Sec. 10]{Kirwan84} shows that the flow is well-behaved on the ambient manifold $M$, which is sufficient to show that Condition 4 holds on the singular space $Z$ (cf. \cite[Prop. 4.2]{wilkin-singular-morse}).

In this paper we take a different approach, and show that if the space $Z$ is locally compact then Condition 4 follows from the analyticity of the function $f : Z \rightarrow \mathbb{R}$. If the function $f : Z \rightarrow \mathbb{R}$ is the restriction of a smooth, but not analytic, function on $M$ then the limit of the flow may not be a single point (see for example \cite[pp13-14]{PalisdeMelo82}) hence the flow mapping a level set to the nearest critical level may not even be defined, let alone continuous, and so in this case extra conditions would be needed to show that the main theorem of Morse theory holds using the methods of \cite{wilkin-singular-morse}.  In contrast to the proof of \cite[Prop. 4.2]{wilkin-singular-morse}, there is no need to study the behaviour of the flow using local coordinates on the ambient manifold $M$, and so the proof given here is intrinsic to the singular space $Z$. Since Conditions 1, 2 and 3 are easy to check for moment maps on varieties, then the theorem below generalises \cite[Prop. 4.2]{wilkin-singular-morse}.

\begin{theorem}\label{thm:locally-compact-4}
Let $Z \subset M$ be a closed locally compact subset and $f : Z \rightarrow \mathbb{R}$ be a function satisfying Conditions 1, 2 and 3. Then Condition 4 is satisfied.
\end{theorem}

As a consequence of the above theorem, we can give a simple criterion for Conditions 1--4 to hold.

\begin{corollary}\label{cor:proper-1-4}
If $f : Z \rightarrow \mathbb{R}$ is proper and analytic, then Conditions 1--4 are satisfied.
\end{corollary}

\section{Proof of the main results}

The local existence of the flow together with Condition 2 shows that if there are no critical values in $[a,b]$ then the flow defines a homeomorphism of level sets $L_a^b : f^{-1}(a) \rightarrow f^{-1}(b)$. If $c$ is a critical value and there are no critical values in $[a,c)$ then Condition 2 and 3 guarantee a well-defined continuous map $L_a^c : f^{-1}(a) \rightarrow f^{-1}(c)$ (this is explained in detail in \cite[Prop. 2.4]{wilkin-singular-morse}). If $C \subset f^{-1}(c)$ is the subset of critical points, then $(L_a^c)^{-1}(C) = W_C^- \cap f^{-1}(a)$ and the restriction $L_a^c : (f^{-1}(a) \setminus W_C^-) \cap f^{-1}(a) \rightarrow f^{-1}(c) \setminus C$ is a homeomorphism. Condition 1 implies that we can always find such values of $a < c$ for each critical value of $c$.

With this setup, we can now prove that Condition 4 holds if the first three conditions hold and the space $Z$ is locally compact.

\begin{proof}[Proof of Theorem \ref{thm:locally-compact-4}]
Let $c$ be a critical value. Since Condition 1 implies that the critical values are isolated, then there exists $a < c$ such that there are no critical values in $[a, c)$.

Suppose that there is a critical point $x$ for which Condition 4 is not satisfied. Since $f$ is analytic, then the Lojasiewicz inequality $\| \grad f(y) \| \geq C \left| f(x) - f(y) \right|^{1-\theta}$ (cf. \cite{Simon83}) applies in a neighbourhood $V$ of the critical point $x$, and the constant $C$ in the inequality is fixed on this neighbourhood. Since $Z$ is locally compact and has topology induced by the metric topology on $M$, then we can choose $\delta > 0$ so that $\{ y \in Z \, : \, \| y - x \| \leq \delta \}$ is compact. By shrinking $\delta$ if necessary we can also guarantee that $\| y - x \| < \delta$ implies that $y \in V$. Now choose $\varepsilon > 0$ so that $a \leq c - \varepsilon$ and that $\frac{1}{C \theta} \varepsilon^\theta < \frac{1}{2} \delta$. If a flow line $y_t$ is contained in the neighbourhood $V$ and satisfies $f(y_0) = f(x) = c$, then the Lojasiewicz inequality implies that
\begin{align*}
\frac{\partial}{\partial t} \left( c - f(y_t) \right)^\theta & = - \theta \left( c - f(y_t) \right)^{\theta - 1} \cdot \frac{\partial}{\partial t} f(y_t) \\
 & = \theta \left( c - f(y_t) \right)^{\theta - 1} \cdot \| \grad f(y_t) \|^2 \\
 & \geq C \theta \| \grad f(y_t) \| .
\end{align*}
Now we show that if $\| y_0 - x \| < \frac{1}{2} \delta$ and $f(y_T) = c - \varepsilon$ then $y_t \in V$ for all $t \in [0,T]$. Let $t_{m}$ be the minimal value of $t$ for which $y_{t_m} \notin V$ and suppose for contradiction that $f(y_{t_m}) > c-\varepsilon$ (and so $t_m < T$). Then for all $0 < t \leq t_m$, the length of the flow line from $y_0 \in f^{-1}(c)$ to $y_t$ is
\begin{equation*}
\int_0^t \| \grad f(y_s) \| \, ds \leq \frac{1}{C \theta} | c - f(y_t) |^\theta < \frac{1}{C \theta} \varepsilon^\theta < \frac{1}{2} \delta .
\end{equation*}
The triangle inequality for the distance on $M$ then shows that $\| y_{t_m} - x \| \leq \| y_{t_m} - y_0 \| + \| y_0 - x \| < \delta$, and so $y_{t_m} \in V$, contradicting the definition of $t_m$. Therefore we must have $f(y_{t_m}) < c-\varepsilon = f(y_T)$ and so the Lojasiewicz inequality applies along the flow line for all $0 \leq t \leq T$, which (after repeating the above argument) implies that $\| y_T - x \| < \delta$. 

Now let $L := L_{c-\varepsilon}^c : f^{-1}(c-\varepsilon) \rightarrow f^{-1}(c)$ be the level set map defined above, and recall that $L^{-1}$ is well-defined and continuous on $f^{-1}(c) \setminus C$. If Condition 4 is not satisfied for the critical point $x$, then there exists an open subset $U \subset f^{-1}(c-\varepsilon)$ containing $W_x^- \cap f^{-1}(c-\varepsilon)$, and a sequence $\{ x_n \}$ in $f^{-1}(c) \setminus C$ such that 
\begin{enumerate}

\item $x_n \rightarrow x$, and

\item the corresponding sequence $\{ L^{-1}(x_n) \}$ in $f^{-1}(c-\varepsilon)$ satisfies $L^{-1}(x_n) \notin U$ for all $n$.

\end{enumerate}

Choose $N$ such that $\| x_n - x \| < \frac{1}{2} \delta$ for all $n \geq N$. The above proof shows that the flow line from $x_n$ to $L^{-1}(x_n)$ has uniformly bounded length $\| L^{-1}(x_n) - x \| < \delta$ for all $n \geq N$, and so the sequence $\{ L^{-1}(x_n) \}$ is contained in a compact set by our choice of $\delta$ (which uses the local compactness of $Z$). Therefore there is a subsequence $\{ x_{n_k} \}$ such that $L^{-1}(x_{n_k}) \rightarrow z \in f^{-1}(c-\varepsilon)$. Since the map $L : f^{-1}(c-\varepsilon) \rightarrow f^{-1}(c)$ is continuous, then
\begin{equation*}
L(z) = \lim_{k \rightarrow \infty} L \circ L^{-1}(x_{n_k}) = \lim_{k \rightarrow \infty} x_{n_k} = x ,
\end{equation*}
so $z \in W_x^-$, contradicting the assumption that $L^{-1}(x_n) \notin U$ for all $n$. Therefore Condition 4 is satisfied on the level set $f^{-1}(c-\varepsilon)$.

Since the map on level sets defines a homeomorphism $f^{-1}(a) \rightarrow f^{-1}(c-\varepsilon)$ then the same is true for the level set $f^{-1}(a)$. Therefore Condition 4 is satisfied for all level sets $f^{-1}(a)$ such that $a < c$ and there are no critical values in the interval $[a,c)$.
\end{proof}

\begin{remark}
The above proof depends on two key facts: (a) the level set map $L : f^{-1}(c-\varepsilon) \rightarrow f^{-1}(c)$ is continuous, and (b) there exists a neighbourhood of $x$ such that flow lines with initial condition in this neighbourhood must remain in a compact subset of $Z \cap f^{-1}([c-\varepsilon, c])$. The first fact follows from the analyticity of the function $f$, and the second follows from the analyticity of $f$ and the local compactness of $Z$. If we only assume that $f$ is the restriction of a smooth function on the ambient manifold $M$, then we would also need some additional assumptions to guarantee that (a) and (b) above hold. 
\end{remark}

\begin{proof}[Proof of Corollary \ref{cor:proper-1-4}]
An analytic function defined on a compact set has isolated critical values (cf. \cite{SoucekSoucek71}). Therefore Condition 1 is satisfied for a proper analytic function.

If the function is proper, then $f^{-1}([c-\varepsilon,c+\varepsilon])$ is compact for all $\varepsilon$, hence the Lojasiewicz inequality method (cf. \cite{Simon83}) shows that the gradient flow of $f$ either converges or flows out of this set, and so Condition 2 is satisfied.

Condition 3 is satisfied by assumption. Since $f$ is proper then $f^{-1}([c-\varepsilon, c])$ is compact, and so the previous theorem applies to show that Condition 4 is also satisfied.
\end{proof}


\end{document}